\documentclass[12pt, reqno]{amsart}
\usepackage{amsmath, amsthm, amscd, amsfonts, amssymb, graphicx, color}
\usepackage[bookmarksnumbered, colorlinks, plainpages]{hyperref}
\hypersetup{colorlinks=true,linkcolor=red, anchorcolor=green, citecolor=cyan, urlcolor=red, filecolor=magenta, pdftoolbar=true}

\newtheorem{theorem}{Theorem}[section]

\theoremstyle{definition}
\newtheorem{definition}[theorem]{Definition}
\newtheorem{example}[theorem]{Example}

\newtheorem{property}[theorem]{Property}
\theoremstyle{remark}

\numberwithin{equation}{section}

\begin{document}

\setcounter{page}{1}

\title[Blowing-up solutions ...]{Blowing-up solutions of the time-fractional dispersive equations}

\author[B. Ahmad, A. Alsaedi, M. Kirane \MakeLowercase{and} B. T. Torebek]{B. Ahmad, A. Alsaedi,  M. Kirane \MakeLowercase{and} Berikbol T. Torebek}

\address{\textcolor[rgb]{0.00,0.00,0.84}{Bashir Ahmad \newline NAAM Research Group, Department of Mathematics, \newline Faculty of Science, King Abdulaziz University, \newline P.O. Box 80203, Jeddah 21589, Saudi Arabia}}
\email{\textcolor[rgb]{0.00,0.00,0.84}{bashirahmad\_qau@yahoo.com}}

\address{\textcolor[rgb]{0.00,0.00,0.84}{Ahmed Alsaedi \newline NAAM Research Group, Department of Mathematics, \newline Faculty of Science, King Abdulaziz University, \newline P.O. Box 80203, Jeddah 21589, Saudi Arabia}}
\email{\textcolor[rgb]{0.00,0.00,0.84}{aalsaedi@hotmail.com}}

\address{\textcolor[rgb]{0.00,0.00,0.84}{Mokhtar Kirane\newline LaSIE, Facult\'{e} des Sciences, \newline Pole Sciences et Technologies, Universit\'{e} de La Rochelle \newline Avenue M. Crepeau, 17042 La Rochelle Cedex, France \newline NAAM Research Group, Department of Mathematics, \newline Faculty of Science, King Abdulaziz University, \newline P.O. Box 80203, Jeddah 21589, Saudi Arabia}}
\email{\textcolor[rgb]{0.00,0.00,0.84}{mkirane@univ-lr.fr}}

\address{\textcolor[rgb]{0.00,0.00,0.84}{Berikbol T. Torebek \newline Al--Farabi Kazakh National University \newline Al--Farabi ave. 71, 050040, Almaty, Kazakhstan \newline Institute of
Mathematics and Mathematical Modeling \newline 125 Pushkin str.,
050010 Almaty, Kazakhstan \newline RUDN University, 6 Miklukho-Maklay St., 117198 Moscow, Russia}}
\email{\textcolor[rgb]{0.00,0.00,0.84}{torebek@math.kz}}

\thanks{All authors contributed equally to the manuscript and read and approved the final manuscript.}

\let\thefootnote\relax\footnote{$^{*}$Corresponding author}



\begin{abstract}This paper is devoted to the study of initial-boundary value problems for  time-fractional analogues of Korteweg-de Vries, Benjamin-Bona-Mahony, Burgers, Rosenau, Camassa-Holm, Degasperis-Procesi, Ostrovsky and time-fractional modified Korteweg-de Vries-Burgers equations on a bounded domain. Sufficient conditions for the blowing-up of solutions in finite time of aforementioned equations are presented. We also discuss the maximum principle and influence of gradient non-linearity on the global solvability of initial-boundary value problems for the time-fractional Burgers equation. The main tool of our study is the Pohozhaev nonlinear capacity method. We also provide some illustrative examples.
\end{abstract} \maketitle
\tableofcontents
\section{Introduction}
Nonlinear wave phenomenon is one of the important areas of scientific investigation. Among the mathematical models describing the dynamics of wave equations include Korteweg-de Vries equation, Burgers equation, Benjamin-Bona-Mahony equation, Rosenau equation and Ostrovsky equation.

Bateman-Burgers equation or Burgers equation \cite{Bat15, Bur48}
\begin{equation}\label{BB}u_t+uu_x=\nu u_{xx},\,\,\nu>0,\end{equation}
is a fundamental partial differential equation occurring in various areas of applied mathematics, such as fluid mechanics, nonlinear acoustics, gas dynamics, traffic flow.

The Korteweg-de Vries equation \cite{KV95} is well known in different fields of science and technology, it reads
\begin{equation}\label{KdV}u_t+uu_x+u_{xxx}=0.\end{equation}

In \cite{BBM72}, Benjamin, Bona and Mahony proposed the following equation to describe long waves on the water surface
\begin{equation}\label{BBM}u_t-u_{txx}+uu_x=0.\end{equation}
In \cite{Ros86} Rosenau suggested the following equation to describe waves on ``shallow'' water:
\begin{equation}\label{Ros}u_t+u_{txxxx}+u_x+uu_x=0.\end{equation}
In \cite{Ost78} Ostrovsky derived an equation for weakly nonlinear surface and internal waves in a rotating ocean
\begin{equation}\label{Ost}u_{tx}+u_{xx}+u_{xxxx}+(uu_x)_x=0.\end{equation}
The Camassa-Holm equation
\begin{equation}\label{CH}u_{t}-u_{txx}+2\kappa u_{x}+3uu_x=2u_xu_{xx}+uu_{xxx},\,\kappa>0,\end{equation}
was introduced by Camassa and Holm \cite{CH93} as a bi-Hamiltonian model for waves in shallow water and the following Degasperis-Procesi equation, one of the important model of mathematical physics
\begin{equation}\label{DP}u_{t}-u_{txx}+2\kappa u_{x}+4uu_x=3u_xu_{xx}+uu_{xxx},\,\kappa>0.\end{equation}
The Korteweg-de Vries-Burgers equation
\begin{equation}\label{KdV-BB}u_t+uu_x+u_{xxx}=\nu u_{xx},\,\,\,\nu>0,\end{equation}
modified Korteweg-de Vries-Burgers equation
\begin{equation}\label{mKdV}u_t+u^2u_x+u_{xxx}=\nu u_{xx},\,\,\,\nu>0,\end{equation}
Benjamin-Bona-Mahony-Burgers equation
\begin{equation}\label{BBM-BB}u_t-u_{txx}+uu_x=\nu u_{xx},\end{equation}
Korteweg-de Vries-Benjamin-Bona-Mahony equation
\begin{equation}\label{KdV-BBM}u_t-u_{txx}+u_{xxx}+uu_x=0,\end{equation}
Rosenau-Burgers equation
\begin{equation}\label{Ros-BB}u_t+u_{txxxx}+u_x+uu_x=\nu u_{xx},\end{equation}
Rosenau-Korteweg-de Vries equation
\begin{equation}\label{Ros-KdV}u_t+u_{txxxx}+u_{xxx}+u_x+uu_x=0,\end{equation}
Rosenau-Benjamin-Bona-Mahony equation
\begin{equation}\label{Ros-BBM}u_t-u_{txx}+u_{txxxx}+u_x+uu_x=0,\end{equation}
have important applications in different physical situations such as waves on shallow water, and processes in semiconductors with differential conductivity \cite{BS76, FPS01, Ros89, Shu87, SG69, Zh05}.

This paper is devoted to blowing-up solutions of time-fractional analogues of the above equations. The approach to the problem
is based on the Pohozhaev nonlinear capacity method \cite{MP98, MP01, MP04}; more precisely, on the choice of test functions according to initial and boundary conditions under
consideration.

Here, we give a simple case of the analysis of a rough blow-up, i.e., the case where the solution tends to infinity as $t \rightarrow T^*$ on $[0,L]$; more exactly, when the integral
$$\int\limits_0^L u(x, t)\phi(x)dx$$ tends to infinity as $t \rightarrow T^*$ for the given function $\phi.$

In \cite{Kor12a, Kor12b, KP13, KY14, KY15} Korpusov et al. obtained sufficient conditions for the finite time blow-up of solutions of initial-boundary problems for Burgers, Korteweg-de Vries, Benjamin-Bona-Mahony and Rosenau type equations. We also note that the blow-up of solutions of the initial problems for the Korteweg-de Vries and critical Korteweg-de Vries equations are investigated in \cite{MM02, MM14, MMR14, Po10, Po10a, Po11, Po11a, Po12, Po12a, Po12b}. Blow-up of solutions of the initial problems for the Ostrovsky equation is proved in \cite{LPS10}.

Descriptions of some physical applications and numerical simulations of the time-fractional dispersive equations are given in \cite{FH18, HHG19, LZR16, LV18, QTWZ17, SBA18, SB12, XA13, Yok18}.

Recently, the study of blowing-up solutions of time-fractional nonlinear partial differential equations received great attention. For example, the authors of this paper obtained results on the blow-up of the solutions of time-fractional Burgers equation \cite{AKT19a, Tor19} and fractional reaction-diffusion equation \cite{AKT19b}. We note that the blow-up of the solution of various nonlinear fractional problems was investigated in \cite{AAAKT15, AAKMA17, CSWSS18, KNS08, Pav18, XX18}.

Let us briefly describe the problems investigated in this paper:
\begin{itemize}
  \item Blowing-up solutions of the time-fractional Rosenau-KdV-BBM-Burgers equation with initial conditions described as follows:
\begin{equation*}
\begin{split}
&\partial_{+0,t}^\alpha (u-au_{xx}+bu_{xxxx})+cu_{xxx}-du_{xx}+u_x+uu_{x}=0,\,0<x<L,\,t>0,\\
&u(x,0)=u_0(x),\,\,\,x\in[0,L],
\end{split}
\end{equation*}
where $a, b, c, d\in \mathbb{R}$ and $u_0$ is a given function.
\item Blowing-up solutions of the initial-boundary problem for the time-fractional Camassa-Holm--Degasperis-Procesi equation
\begin{equation*}
\begin{split}
&\partial_{+0,t}^\alpha (u-u_{xx})+au_{x}+buu_{x}-cu_xu_{xx}-duu_{xxx}=0,\,0<x<L,\,t>0,\\
&u(x,0)=u_0(x),\,\,\,x\in[0,L],
\end{split}
\end{equation*}
where $a, b, c, d\in \mathbb{R}$ and $u_0$ is a given function.
\item Blowing-up solutions of the time-fractional Ostrovsky equation with initial conditions:
\begin{equation*}
\begin{split}
&\partial_{+0,t}^\alpha u_x+au_{xx}+bu_{xxxx}+(uu_{x})_x=0,\,0<x<L,\,t>0,\\
&u(x,0)=u_0(x),\,\,\,x\in[0,L],
\end{split}
\end{equation*}
where $a, b\in \mathbb{R}$ and $u_0$ is a given function.
  \item Blowing-up solutions of the initial problem for the time-fractional analogue of the modified Korteweg-de Vries-Burgers equation with dissipation:
\begin{equation*}\begin{split}
& \partial^\alpha_{+0,t}u+u^2u_x+au_{xxx}-bu_{xx}=0,\,\,\,x\in (0,L),\,t>0,\\
& u(x,0)=u_0(x),\,\,\,x\in[0,L],
\end{split}
\end{equation*}
where $a,b\in\mathbb{R}$ and $u_0$ is a sufficiently smooth function.
  \item Maximum principle and gradient blow-up in time-fractional Burgers equation
\begin{equation*}
\partial^\alpha_{+0,t}u+uu_x=\nu u_{xx},\,\,\,x\in (0,L),\,t>0,
\end{equation*} with an initial condition
\begin{equation*}
u(x,0)=u_0(x),\,\,\,x\in[0,L],
\end{equation*}
where $\nu>0$ and $u_0$ is a sufficiently smooth function.
\end{itemize}

\subsection{Preliminaries}
\subsubsection{Fractional operators}
\label{Pre}
Here, we recall definitions and properties of fractional order integral and differential operators \cite{KST06, N03, SKM87}.

\begin{definition} \cite{KST06} (Riemann-Liouville integral). Let $f$ be a locally integrable real-valued function on $-\infty\leq a<t<b\leq+\infty.$ The Riemann--Liouville fractional integral $I_{+a} ^\alpha$ of order $\alpha\in\mathbb R$ ($\alpha>0$) is defined as
$$
I_{+a} ^\alpha  f\left( t \right) = \left(f*K_{\alpha}\right)(t) =\frac{1}{{\Gamma \left( \alpha \right)}}\int\limits_a^t {\left(
{t - s} \right)^{\alpha  - 1} f\left( s \right)} ds,$$
where $K_{\alpha}(t)=\frac{t^{\alpha-1}}{\Gamma(\alpha)},$ $\Gamma$ denotes the Euler gamma function.
\end{definition}
The convolution here will be understood in the sense of the above definition.

\begin{definition} \cite{KST06} (Riemann-Liouville derivative). Let $f\in L^1([a,b]),$ $-\infty\leq a<t<b\leq+\infty$ and $f*K_{m-\alpha}(t)\in W^{m,1}([a,b]),\, m=[\alpha]+1,\, \alpha>0,$ where $W^{m,1}([a,b])$ is the Sobolev space defined as $$W^{m,1}([a,b])=\left\{f\in L^1([a,b]):\,\frac{d^m}{dt^m}f\in L^1([a,b])\right\}.$$ The Riemann--Liouville fractional derivative $D_{+a} ^\alpha$ of order $\alpha>0$ ($m-1<\alpha<m,\,m\in \mathbb{N}$) is defined as
$$D_{+a} ^\alpha f \left( t \right) = \frac{{d^m }}{{dt^m }}I_{+a} ^{m - \alpha } f \left( t \right)={\rm{}}\frac{1}{{\Gamma \left( m-\alpha \right)}}\frac{d^m}{dt^m}\int\limits_a^t {\left({t - s} \right)^{m-1-\alpha} f\left( s \right)} ds.$$
\end{definition}

\begin{definition}\label{def3} \cite{KST06} (Caputo derivative). Let $f\in L^1([a,b]),$ $-\infty\leq a<t<b\leq+\infty$ and $f*K_{m-\alpha}(t)\in W^{m,1}([a,b]),\, m=[\alpha],\, \alpha>0.$ The Caputo fractional derivative $\partial_{+a}^\alpha$ of order $\alpha\in\mathbb R$ ($m-1<\alpha<m,\,m\in \mathbb{N}$) is defined as
\begin{align*}\partial_{+a} ^\alpha  f \left( t \right)= D_{+a}
^\alpha  \left[ f\left( t \right) - f\left( a \right) -f'\left( a \right)\frac{(t-a)}{1!}-... - f^{(m-1)}\left( a \right)\frac{(t-a)^{m-1}}{(m-1)!}\right].\end{align*}

If $f\in C^m([a,b])$,  then the Caputo fractional derivative $\partial_{+a}^\alpha$ of order $\alpha\in\mathbb R$ ($m-1<\alpha<m,\,m\in \mathbb{N}$) is defined as $$\partial_{+a}^\alpha  \left[ f \right]\left( t \right) = I_{+a} ^{m - \alpha } f^{(m)}\left( t \right)=\frac{1}{{\Gamma \left( m-\alpha \right)}}\int\limits_a^t {\left({t - s} \right)^{m-1-\alpha} f^{(m)}\left( s \right)} ds.$$
\end{definition}
\begin{property}\cite{AAK17} Let $0<\alpha\leq1,$ $f\in C([0,T]),\,\, f'\in L^1([0,T])$ and $u$ be monotone. Then
\begin{equation}\label{Ineq}
2f(t)\partial^\alpha_{+0,t}f(t)\geq \partial^\alpha_{+0,t}f^2(t),\,\,\,t\in (0,T].
\end{equation}
\end{property}
\begin{property}\label{Pr1}\cite{Lu09}
Let  $f \in C^1((0, T )) \cap C([0, T ])$ attain its maximum over the interval $[0, T ]$ at $t_0 \in (0, T ].$
Then $\partial^\alpha_{+0,t} f(t_0)\geq 0.$\\
Let  $f \in C^1((0, T )) \cap C([0, T ])$ attain its minimum over the interval $[0, T ]$ at $t_0 \in (0, T ].$
Then $\partial^\alpha_{+0,t} f(t_0)\leq 0.$
\end{property}
\subsubsection{Finite time blow-up of solutions of a fractional differential equation} We consider the fractional differential equation
\begin{equation}\label{01}
\begin{split}
&\partial_{0+}^\alpha u(t)=u^2(t),\,\,t>0,\, 0<\alpha<1,\\
&u(0)=u_0\in \mathbb{R}.
\end{split}
\end{equation}

The blow-up of solutions to \eqref{01} is assured by the following theorem.
\begin{theorem}\label{th01}\cite{HKL14} If $u_0 > 0,$ then the solution of problem \eqref{01} blows-up in a finite time \begin{equation}\label{T*}\left(\frac{\Gamma(\alpha+1)}{4u_0}\right)^{\frac{1}{\alpha}}\leq T^*\leq \left(\frac{\Gamma(\alpha+1)}{u_0}\right)^{\frac{1}{\alpha}},\end{equation}
that is $\lim\limits_{t\rightarrow T^*}u(t)=+\infty.$\end{theorem}

\section{Blowing-up solutions of the time-fractional Rosenau-KdV-BBM-Burgers equation}
In this section we consider the time-fractional Rosenau-KdV-BBM-Burgers equation:
\begin{equation}\label{FBB-1}
\begin{split}
&\partial_{+0,t}^\alpha (u-au_{xx}+bu_{xxxx})+cu_{xxx}-du_{xx}+u_x+uu_{x}=0,\,0<x<L,\,t>0,\\
&u(x,0)=u_0(x),\,\,\,x\in[0,L],
\end{split}
\end{equation}
where $a, b, c, d\in \mathbb{R}$ and $u_0$ is a given function.

The equation \eqref{FBB-1} is called the Rosenau-KdV-BBM-Burgers equation with time-fractional derivative as it is a generalization of the following well-known equations:
\begin{itemize}
  \item If $\alpha=1$ and $a=b=c=0,\, d>0,$ then the equation \eqref{FBB-1} coincides with the classical Burgers equation \eqref{BB};
  \item If $\alpha=1$ and $a=b=d=0,\, c=1,$ then the equation \eqref{FBB-1} coincides with the classical KdV equation \eqref{KdV};
  \item If $\alpha=1$ and $b=c=d=0,\, a=1,$ then the equation \eqref{FBB-1} coincides with the classical BBM equation \eqref{BBM};
  \item If $\alpha=1$ and $a=b=0,\, c=1,\, d>0,$ then the equation \eqref{FBB-1} coincides with the classical KdV-Burgers equation \eqref{KdV-BB};
  \item If $\alpha=1$ and $b=c=0,\, a=1,\, d>0,$ then the equation \eqref{FBB-1} coincides with the classical BBM-Burgers equation \eqref{BBM-BB};
  \item If $\alpha=1$ and $b=d=0,\, a=c=1,$ then the equation \eqref{FBB-1} coincides with the classical KdV-BBM equation \eqref{KdV-BBM};
  \item If $\alpha=1$ and $a=c=d=0,\, b=1,$ then equation \eqref{FBB-1} coincides with the classical Rosenau equation \eqref{Ros};
  \item If $\alpha=1$ and $a=c=0,\, b=1,$ then equation \eqref{FBB-1} coincides with the classical Rosenau-Burgers equation \eqref{Ros-BB};
  \item If $\alpha=1$ and $c=d=0,\, a=b=1,$ then equation \eqref{FBB-1} coincides with the classical Rosenau-BBM equation \eqref{Ros-BBM};
  \item If $\alpha=1$ and $a=d=0,\, b=c=1,$ then equation \eqref{FBB-1} coincides with the classical Rosenau-KdV equation \eqref{Ros-KdV}.
\end{itemize}

We study the question of the blow-up of a classical solution $u\in C^{1,4}_{t,x}([0,T]\times[0,L])$ of problem \eqref{FBB-1}.

Let us  consider a function $\varphi\in C^4([0, L])$  and suppose that  the solution $u\in C^{1,4}_{t,x}([0,T]\times[0,L])$ of problem \eqref{FBB-1} exists.
Multiplying equation \eqref{FBB-1} by $\varphi$ and integrating by parts, we obtain
\begin{equation}\label{FBB-2}
\begin{split}
  \partial^\alpha_{+0,t}&\int\limits_0^Lu(x,t) (\varphi(x)-a\varphi''(x)+b\varphi''''(x))dx \\& = \int\limits_0^Lu(x,t)(c\varphi'''(x)+d\varphi''(x)+\varphi'(x))dx +\frac{1}{2}\int\limits_0^Lu^2(x,t)\varphi'(x)dx\\& +\mathcal{B}(u(L,t),\varphi(L))-\mathcal{B}(u(0,t),\varphi(0)),
\end{split}
\end{equation} where
\begin{align*}\mathcal{B}(u(x,t),\varphi(x))&=a \partial^\alpha_{+0,t}u_x(x,t)\varphi(x)-a \partial^\alpha_{+0,t}u(x,t)\varphi'(x)\\&-b\partial^\alpha_{+0,t}u_{xxx}(x,t)\varphi(x) +b\partial^\alpha_{+0,t}u_{xx}(x,t)\varphi'(x)\\& -b\partial^\alpha_{+0,t}u_x(x,t)\varphi''(x)+b\partial^\alpha_{+0,t}u(x,t)\varphi'''(x)\\& -cu_{xx}(x,t)\varphi(x) +cu_x(x,t)\varphi'(x)-cu(x,t)\varphi''(x)\\&+du_x(x,t)\varphi(x)-du(x,t)\varphi'(x)\\& -u(x,t)\varphi(x)-\frac{1}{2}u^2(x,t)\varphi(x).\end{align*}
Let the function $\varphi(x)$ be monotonically nondecreasing:
\begin{equation}\label{FBB-3}
\varphi'(x)\geq 0\,\,\,\,\textrm{for}\,\,\,\,x\in[0,L]
\end{equation}
and satisfy the following properties
\begin{equation}\label{FBB-4}
\left\{\begin{split}
  & \theta_1:=\frac{1}{2}\int\limits_0^L\frac{(c\varphi'''(x)+d\varphi''(x)+\varphi'(x))^2}{\varphi'(x)}dx<\infty;\\
  & \theta_2:=2\int\limits_0^L\frac{(\varphi(x)-a\varphi''(x)+b\varphi''''(x))^2}{\varphi'(x)}dx<\infty.
\end{split}\right.
\end{equation}
Then we have
\begin{align*}
  2&\int\limits_0^Lu(x,t)(c\varphi'''(x)+d\varphi''(x)+\varphi'(x))dx +\int\limits_0^Lu^2(x,t)\varphi'(x)dx\\
  &=\int\limits_0^Lv^2(x,t)\varphi'(x)dx-\int\limits_0^L\frac{(c\varphi'''(x)+d\varphi''(x)+\varphi'(x))^2}{\varphi'(x)}dx,
\end{align*} where $$v(x,t)=u(x,t)+\frac{c\varphi'''(x)+d\varphi''(x)+\varphi'(x)}{\varphi'(x)}.$$

Using the H\"{o}lder inequality, we obtain the following estimate
\begin{align*}&\left(\int\limits_0^Lv(x,t)(\varphi(x)-a\varphi''(x)+b\varphi''''(x))dx\right)^2\\& \leq\int\limits_0^Lv^2(x,t)\varphi'(x)dx\int\limits_0^L\frac{\left(\varphi(x)-a\varphi''(x)+b\varphi''''(x)\right)^2}{\varphi'(x)}dx.\end{align*}
Then, expression \eqref{FBB-2} takes the form
\begin{equation}\label{FBB-5}
\partial_{+0,t}^\alpha F(t)\geq \theta_2^{-1}F^2(t)+\Phi(t)-\theta_1,
\end{equation} where $$F(t)=\int\limits_0^Lv(x,t)\left(\varphi(x)-a\varphi''(x)+b\varphi''''(x)\right)dx$$ and $$\Phi(t)=\mathcal{B}(u(L,t),\varphi(L))-\mathcal{B}(u(0,t),\varphi(0)).$$
Then the following theorem holds.
\begin{theorem}\label{thBB} Let $u_0(x)\in L^1([0,L])$ and the solution $u$ of the equation \eqref{FBB-1} is such that $u\in C^{1,4}_{t,x}\left((0,L)\times(0,T)\right)$ and let the function $\varphi$ satisfy conditions \eqref{FBB-3}, \eqref{FBB-4}. If $\Phi(t)-\theta_1\geq 0,\,\,\textrm{for all}\,\, t>0,$ and $F(0)>0,$ then $$F(t)\rightarrow +\infty\,\,\, \textrm{for} \,\,\,t\rightarrow T^*,$$ where $T^*$ satisfies estimate \eqref{T*}.
\end{theorem}
\begin{proof} Obviously
\begin{equation*} \partial_{+0,t}^\alpha \tilde{F}(t)\geq \tilde{F}^2(t),\end{equation*} where $\tilde{F}(t)=\theta_2 F(t).$

Since the function $\tilde{F}(t)$ is an upper solution of equation \eqref{01}, therefore $\tilde{F}(t)\rightarrow +\infty$ for $t\rightarrow T^*,$ where $T^*$ satisfies estimate \eqref{T*}. Whereupon $F(t)\rightarrow +\infty$ for $t\rightarrow T^*.$
\end{proof}
Note that the trial function method has great practical convenience.
\begin{example}(Fractional Korteweg-de Vries equation).  Consider the problem \eqref{FBB-1} with $a=b=d=0, c=1$ on the interval $[0, 1]$ equipped with the  boundary conditions:
\begin{align*} &u(0,t)=0,\,\,t\geq 0,\\&u(1,t)=0,\,\,t\geq 0,\\& u_x(1,t)=u_x(0,t)+u_{xx}(1,t),\,\,t\geq 0.\end{align*} Letting $\varphi(x)=x$, we obtain
$$\theta_1:=0,\,\,\theta_2:=\frac{1}{6}$$ and \begin{align*} \Phi(t)=\theta_1=0,\,\,\textrm{for all}\,\,\,\, t>0;\end{align*} hence it follows from Theorem \ref{thBB} that the solution of problem \eqref{FBB-1} blows up in finite time under the condition
$$\int\limits_0^1u_0(x)xdx>0.$$
\end{example}
\begin{example} (Fractional Burgers equation). Let $a=b=c=0,\,d>0$ in problem \eqref{FBB-1} on the interval $[0, 1]$ and let the solution of problem \eqref{FBB-1} satisfy the Robin type nonlinear boundary conditions:
\begin{align*} &u(0,t)=0,\,\,t\geq 0,\\& du_x(1,t)-du(1,t)-u(1,t)-\frac{1}{2}u^2(1,t)=0,\,\,t\geq 0.\end{align*} Then, if $\varphi(x)=x$, we obtain
$$\theta_1:=0,\,\,\theta_2:=\frac{1}{6}$$ and \begin{align*} \Phi(t)= \theta_1=0,\,\,\textrm{for all}\,\,\,\, t>0;\end{align*} hence it follows from Theorem \ref{thBB} that the solution of problem \eqref{FBB-1} blows up in finite time under the condition
$$\int\limits_0^1u_0(x)xdx>0.$$
\end{example}
\begin{example} (Fractional Benjamin-Bona-Mahony equation). With  $b=c=d=0,\,a=1$,  consider the problem \eqref{FBB-1} on the interval $[0, 1]$ supplemented with final boundary conditions:
\begin{align*} &u(1,t)=0,\,\,t\geq 0,\\& u_x(1,t)=0,\,\,t\geq 0.\end{align*} Taking $\varphi(x)=x^4,$ we obtain
$$\theta_1:=0,\,\,\theta_2:=\frac{395}{48}$$ and \begin{align*} \Phi(t)= 0,\,\,\textrm{for all}\,\,\,\, t>0;\end{align*} hence it follows from Theorem \ref{thBB} that the solution of problem \eqref{FBB-1} blows up in finite time under the condition
$$\int\limits_0^1u_0(x)x^2(x^2-12)dx>0.$$
\end{example}
\begin{example}(Fractional Rosenau equation). Let $a=c=d=0$ and consider problem \eqref{FBB-1} on the interval $[0, 1]$ with Dirichlet type boundary conditions
\begin{align*} &u(0,t)=0,\,\,t\geq 0,\\& u(1,t)=0,\,\,t\geq 0,\\& u_{xx}(1,t)=0,\,\,t\geq 0,\\&\partial^\alpha_{+0,t}u_{xxx}(0,t)-\partial^\alpha_{+0,t}u_{xx}(0,t)=f(t),\,\,t\geq 0.\end{align*} Suppose that $f(t)\geq \frac{1}{2},\,\,\textrm{for all}\,\,t>0.$ Then, if $\varphi(x)=x-1,$ we obtain
$$\theta_1:=\frac{1}{2},\,\,\theta_2:=\frac{2}{3}$$ and \begin{align*} \Phi(t)-\theta_1=f(t)-\frac{1}{2}\geq 0,\,\,\textrm{for all}\,\,\,\, t>0;\end{align*} hence it follows from Theorem \ref{thBB} that the solution of problem \eqref{FBB-1} blows up in finite time under the condition
$$\int\limits_0^1u_0(x)(x-1)dx>\frac{1}{2}.$$
\end{example}
\begin{example}(Fractional Rosenau-Burgers equation). Let $a=c=0$ and consider problem \eqref{FBB-1} on the interval $[0, 1]$ with nonlocal dynamical boundary conditions
\begin{align*} &u(1,t)=0,\,\,t\geq 0,\\& u_x(1,t)+u(0,t)=0,\,\,t\geq 0,\\& u_{xx}(0,t)=0,\,\,t\geq 0,\\&\partial^\alpha_{+0,t}u_{xxx}(1,t)-\partial^\alpha_{+0,t}u_{xx}(1,t)=\frac{1}{2},\,\,t\geq 0.\end{align*} Letting $\varphi(x)=x,$ we obtain
$$\theta_1:=\frac{1}{2},\,\,\theta_2:=\frac{2}{3}$$ and \begin{align*} \Phi(t)-\theta_1=0,\,\,\textrm{for all}\,\,\,\, t>0;\end{align*} hence it follows from Theorem \ref{thBB} that the solution of problem \eqref{FBB-1} blows up in finite time under the condition
$$\int\limits_0^1u_0(x)xdx>-\frac{1}{2}.$$
\end{example}

\section{Blowing-up solutions of the time-fractional Camassa-Holm--Degasperis-Procesi equation}
In this section we consider the time-fractional Camassa-Holm--Degasperis-Procesi equation:
\begin{equation}\label{CH-1}
\begin{split}
&\partial_{+0,t}^\alpha (u-u_{xx})+au_{x}+buu_{x}-cu_xu_{xx}-duu_{xxx}=0,\,0<x<L,\,t>0,\\
&u(x,0)=u_0(x),\,\,\,x\in[0,L],
\end{split}
\end{equation}
where $a, b, c, d\in \mathbb{R}$ and $u_0$ is a given function.

The equation \eqref{CH-1} is called the Camassa-Holm--Degasperis-Procesi equation with time-fractional derivative as it is a generalization of the following well-known equations:
\begin{itemize}
  \item If $\alpha=1$ and $a=2\kappa>0,\,b=3,\, c=2,\, d=1,$ then the equation \eqref{CH-1} coincides with the classical Camassa-Holm equation \eqref{CH};
  \item If $\alpha=1$ and $a=2\kappa>0,\,b=4,\, c=3,\, d=1,$ then the equation \eqref{CH-1} coincides with the classical Degasperis-Procesi equation \eqref{DP}.
\end{itemize}

We study the question of the blow-up of a classical solution $u\in C^{1,3}_{t,x}([0,T]\times[0,L])$ of problem \eqref{CH-1}.
Let us  consider a function $\varphi\in C^3([0, L])$  and suppose that  the solution $u\in C^{1,3}_{t,x}([0,T]\times[0,L])$ of problem \eqref{CH-1} exists. Using the equality $$(u^2)_{xxx} = 6u_xu_{xx}+2uu_{xxx},$$ we reduce the equation \eqref{CH-1} to the equation
\begin{equation}\label{CH-2}
\partial_{+0,t}^\alpha (u-u_{xx})+au_{x}+buu_{x}+(3d-c)u_xu_{xx}-\frac{d}{2}(u^2)_{xxx}=0,\,0<x<L,\,t>0.
\end{equation}

Multiplying equation \eqref{CH-1} by $\varphi$ and integrating by parts, we obtain
\begin{equation}\label{CH-3}
\begin{split}
  \partial^\alpha_{+0,t}&\int\limits_0^Lu(x,t) (\varphi(x)-\varphi''(x))dx \\& = a\int\limits_0^Lu(x,t)\varphi'(x)dx +\frac{3d-c}{2}\int\limits_0^Lu_x^2(x,t)\varphi'(x)dx\\& + \frac{1}{2}\int\limits_0^Lu^2(x,t)(b\varphi'(x)-d\varphi'''(x))dx \\& +\mathcal{B}(u(L,t),\varphi(L))-\mathcal{B}(u(0,t),\varphi(0)),
\end{split}
\end{equation} where
\begin{align*}\mathcal{B}(u(x,t),\varphi(x))&= \partial^\alpha_{+0,t}u_x(x,t)\varphi(x)- \partial^\alpha_{+0,t}u(x,t)\varphi'(x)\\&-au(x,t)\varphi(x)-\frac{b}{2}u^2(x,t)\varphi(x) \\&-\frac{d-c}{2}u_x^2(x,t)\varphi(x) +du(x,t)u_{xx}(x,t)\varphi(x)\\&-du(x,t)u_x(x,t)\varphi'(x)+\frac{d}{2}u^2(x,t)\varphi''(x).\end{align*}
Let $3d-c\geq 0$ and the function $\varphi(x)$ be monotonically nondecreasing:
\begin{equation}\label{CH-4}
\varphi'(x)\geq 0\,\,\,\,\textrm{for}\,\,\,\,x\in[0,L],
\end{equation}
then from \eqref{CH-3} we have
\begin{equation}\label{CH-5}
\begin{split}
  \partial^\alpha_{+0,t}&\int\limits_0^Lu(x,t) (\varphi(x)-\varphi''(x))dx \geq a\int\limits_0^Lu(x,t)\varphi'(x)dx \\& + \frac{1}{2}\int\limits_0^Lu^2(x,t)(b\varphi'(x)-d\varphi'''(x))dx \\& +\mathcal{B}(u(L,t),\varphi(L))-\mathcal{B}(u(0,t),\varphi(0)).
\end{split}
\end{equation}

Let $\varphi$ satisfy the following properties
\begin{equation}\label{CH-6}
\left\{\begin{split}
  & b\varphi'(x)-d\varphi'''(x)\geq 0,\,x\in[0,L];\\
  & \theta_1:=\frac{1}{2}\int\limits_0^L\frac{a^2\varphi'^2(x)}{b\varphi'(x)-d\varphi'''(x)}dx<\infty;\\
  & \theta_2:=2\int\limits_0^L\frac{(\varphi(x)-\varphi''(x))^2}{b\varphi'(x)-d\varphi'''(x)}dx<\infty.
\end{split}\right.
\end{equation}
Then we have
\begin{align*}
  2a&\int\limits_0^Lu(x,t)\varphi'(x)dx +\int\limits_0^Lu^2(x,t)(b\varphi'(x)-d\varphi'''(x))dx\\
  &=\int\limits_0^Lv^2(x,t)(b\varphi'(x)-d\varphi'''(x))dx-a^2\int\limits_0^L\frac{\varphi'^2(x)}{b\varphi'(x)-d\varphi'''(x)}dx,
\end{align*} where $$v(x,t)=u(x,t)+a\frac{\varphi'(x)}{b\varphi'(x)-d\varphi'''(x)}.$$

Using the H\"{o}lder inequality, we obtain the following estimate
\begin{align*}&\left(\int\limits_0^Lv(x,t)(\varphi(x)-\varphi''(x))dx\right)^2\\& \leq\int\limits_0^Lv^2(x,t)(b\varphi'(x)-d\varphi'''(x))dx\int\limits_0^L\frac{\left(\varphi(x)- \varphi''(x)\right)^2}{b\varphi'(x)-d\varphi'''(x)}dx.\end{align*}
Then, expression \eqref{CH-3} takes the form
\begin{equation}\label{CH7}
\partial_{+0,t}^\alpha F(t)\geq \theta_2^{-1}F^2(t)+\Phi(t)-\theta_1,
\end{equation} where $$F(t)=\int\limits_0^Lv(x,t)\left(\varphi(x)-\varphi''(x)\right)dx$$ and $$\Phi(t)=\mathcal{B}(u(L,t),\varphi(L))-\mathcal{B}(u(0,t),\varphi(0)).$$
Then the following theorem holds.
\begin{theorem}\label{thCH} Let $u_0(x)\in L^1([0,L])$ and the solution $u$ of the equation \eqref{CH-1} is such that $u\in C^{1,3}_{t,x}\left((0,L)\times(0,T)\right)$ and let the function $\varphi$ satisfy conditions \eqref{CH-4}, \eqref{CH-6}. If $\Phi(t)-\theta_1\geq 0,\,\,\textrm{for all}\,\, t>0,$ and $F(0)>0,$ then $$F(t)\rightarrow +\infty\,\,\, \textrm{for} \,\,\,t\rightarrow T^*,$$ where $T^*$ satisfies estimate \eqref{T*}.
\end{theorem}
The Theorem \ref{thCH} can be proved as Theorem \ref{thBB}.

Below we give some examples.
\begin{example}(Fractional Camassa-Holm equation).  Consider the problem \eqref{CH-1} with $a=2\kappa>0,\,b=3,\, c=2,\, d=1,$ on the interval $[0, 1]$ equipped with the dynamical boundary conditions:
\begin{align*} &u(0,t)=0,\,\,t\geq 0,\\&u(1,t)=0,\,\,t\geq 0,\\& \partial^\alpha_{+0,t}u_x(1,t)+\frac{1}{2}u^2_x(1,t)=f(t)\geq \frac{2\kappa^2}{3},\,\,t\geq 0.\end{align*} Letting $\varphi(x)=x$, we obtain
$$\theta_1:=\frac{2\kappa^2}{3},\,\,\theta_2:=\frac{2}{9}$$ and \begin{align*} \Phi(t)-\theta_1\geq 0,\,\,\textrm{for all}\,\,\,\, t>0;\end{align*} hence it follows from Theorem \ref{thCH} that the solution of problem \eqref{CH-1} blows up in finite time under the condition
$$\int\limits_0^1u_0(x)xdx>-\frac{\kappa}{3}.$$
\end{example}
\begin{example} (Fractional Degasperis-Procesi). Let $a=2\kappa>0,\,b=4,\, c=3,\, d=1,$ in problem \eqref{CH-1} on the interval $[0, 1]$ and let the solution of problem \eqref{CH-1} satisfy the nonlinear nonlocal boundary conditions:
\begin{align*} &u(1,t)=0,\,\,t\geq 0,\\&u_x(0,t)=0,\,\,t\geq 0, \\& \partial^\alpha_{+0,t}u_x(1,t)+\partial^\alpha_{+0,t}u(0,t)+u^2_x(1,t)=g(t)\geq \frac{\kappa^2}{2},\,\,t\geq 0.\end{align*} Then, if $\varphi(x)=x$, we obtain
$$\theta_1:=\frac{\kappa^2}{2},\,\,\theta_2:=\frac{1}{6}$$ and \begin{align*} \Phi(t)-\theta_1\geq 0,\,\,\textrm{for all}\,\,\,\, t>0;\end{align*} hence it follows from Theorem \ref{thCH} that the solution of problem \eqref{CH-1} blows up in finite time under the condition
$$\int\limits_0^1u_0(x)xdx>-\frac{\kappa}{4}.$$
\end{example}

\section{Blowing-up solutions of the time-fractional Ostrovsky equation}
We consider the equation
\begin{equation}\label{Ost-1}
\partial^\alpha_{+0,t}u_x+au_{xx}+bu_{xxxx}+(uu_x)_x=0,\,\,\,x\in (0,L),\,t>0,
\end{equation} with Cauchy data
\begin{equation}\label{Ost-2}
u(x,0)=u_0(x),\,\,\,x\in[0,L],
\end{equation}
where $a,b\in\mathbb{R}$ and $u_0$ is a sufficiently smooth function.

Multiplying equation \eqref{Ost-1} by a function $\varphi(x)\in C^4([0,L])$ and integrating by parts, we obtain
\begin{equation}\label{Ost-3}
\begin{split}
  &\partial^\alpha_{+0,t}\int\limits_0^Lu(x,t)\varphi'(x)dx =\frac{1}{2}\int\limits_0^Lu^2(x,t)\varphi''(x)dx\\& +\int\limits_0^Lu(x,t)(a\varphi''(x)+b\varphi''''(x))dx+\mathcal{B}(u(x,t),\varphi(x))\Big{|}_0^L,
\end{split}
\end{equation} where \begin{align*}\mathcal{B}(u(x,t),\varphi(x))&=\partial^\alpha_{+0,t}u(x,t)\varphi(x)-au_x(x,t)\varphi(x)+au(x,t)\varphi'(x) \\&-bu_{xxx}\varphi(x)+bu_{xx}(x,t)\varphi'(x)-bu_x(x,t)\varphi''(x)\\&+bu(x,t)\varphi'''(x) -u(x,t)u_x(x,t)\varphi(x)+\frac{1}{2}u^2(x,t)\varphi'(x).\end{align*}

Let the function $\varphi(x)$ satisfy the properties:
\begin{equation}\label{Ost-4}
\varphi''(x)\geq 0\,\,\,\,\textrm{for}\,\,\,\,x\in[0,L]
\end{equation}
and
\begin{equation}\label{Ost-5}
\left\{\begin{split}
  & \theta_1:=\frac{1}{2}\int\limits_0^L\frac{(a\varphi''(x)+b\varphi''''(x))^2}{\varphi''(x)}dx<\infty;\\
  & \theta_2:=2\int\limits_0^L\frac{\varphi'^2(x)}{\varphi''(x)}dx<\infty.
\end{split}\right.
\end{equation}
Then we have
\begin{align*}
  2&\int\limits_0^Lu(x,t)(a\varphi''(x)+b\varphi''''(x))dx +\int\limits_0^Lu^2(x,t)\varphi''(x)dx\\
  &=\int\limits_0^Lv^2(x,t)\varphi''(x)dx-\int\limits_0^L\frac{(a\varphi''(x)+b\varphi''''(x))^2}{\varphi''(x)}dx,
\end{align*} where $$v(x,t)=u(x,t)+\frac{a\varphi''(x)+b\varphi''''(x)}{\varphi''(x)}.$$

Using the H\"{o}lder inequality, we obtain the following estimate
\begin{align*}&\left(\int\limits_0^Lv(x,t)\varphi'(x)dx\right)^2 \leq\int\limits_0^Lv^2(x,t)\varphi''(x)dx\int\limits_0^L\frac{\varphi'^2(x)}{\varphi''(x)}dx.\end{align*}
Then, the expression \eqref{Ost-3} can be rewritten as
\begin{equation}\label{Ost-6}
\partial_{+0,t}^\alpha F(t)\geq \theta_2^{-1}F^2(t)+\Phi(t)-\theta_1,
\end{equation} where $$F(t)=\int\limits_0^Lv(x,t)\varphi'(x)dx$$ and $$\Phi(t)=\mathcal{B}(u(L,t),\varphi(L))-\mathcal{B}(u(0,t),\varphi(0)).$$
Then the following theorem holds.
\begin{theorem}\label{thOst} Let $u_0(x)\in L^1([0,L])$ and the solution $u$ of the problem \eqref{Ost-1}, \eqref{Ost-2} is such that $u\in C^{1,4}_{t,x}\left((0,L)\times(0,T)\right)$  and let the function $\varphi$ satisfy conditions \eqref{Ost-4}, \eqref{Ost-5}. If $\Phi(t)-\theta_1\geq 0,\,\,\textrm{for all}\,\, t>0,$ and $F(0)>0,$ then $$F(t)\rightarrow +\infty\,\,\, \textrm{for} \,\,\,t\rightarrow T^*,$$ where $T^*$ satisfies estimate \eqref{T*}.
\end{theorem}
The Theorem \ref{thOst} can be proved as Theorem \ref{thBB}.
\begin{example} With $a=1$ and $b=-1,$ consider problem \eqref{Ost-1}, \eqref{Ost-2} on the interval $[0, 1]$ subject to Dirichlet type boundary conditions
\begin{align*} &u(0,t)=0,\,\,t\geq 0,\\& u(1,t)=0,\,\,t\geq 0,\\& u_{x}(0,t)=0,\,\,t\geq 0,\\&u_{xxx}(1,t)-2u_{xx}(1,t)=f(t),\,\,t\geq 0.\end{align*} Suppose that $f(t)\geq 1,\,\,\textrm{for all}\,\,t>0.$ Then, if $\varphi(x)=x^2,$ we obtain
$$\theta_1:=1,\,\,\theta_2:=\frac{4}{3}$$ and \begin{align*} \Phi(t)-\theta_1=f(t)-1\geq 0,\,\,\textrm{for all}\,\,\,\, t>0;\end{align*} hence it follows by Theorem \ref{thOst} that the solution of problem \eqref{Ost-1}, \eqref{Ost-2} blows up in finite time under the condition
$$\int\limits_0^1u_0(x)x^2dx>-\frac{1}{3}.$$
\end{example}
\section{Blowing-up solutions of the time-fractional modified KdV-Burgers equation}
Consider the initial value problem for the time-fractional analogue of the well-known modified Korteweg-de Vries-Burgers equation with dissipation:
\begin{equation}\label{mKdV-1}
\partial^\alpha_{+0,t}u+u^2u_x+au_{xxx}-bu_{xx}=0,\,\,\,x\in (0,L),\,t>0,
\end{equation}
\begin{equation}\label{mKdV-2}
u(x,0)=u_0(x),\,\,\,x\in[0,L],
\end{equation}
where $a,b\in\mathbb{R}$ and $u_0$ is a sufficiently smooth function.

Let a function $\varphi\in C^3([0,L])$ satisfy the  properties: \begin{equation}\label{mKdV-*}\varphi(x)\leq 0,\,\, \varphi'(x)\geq 0\,\,\, \textrm{for} \,\,\,x\in [0,L],\end{equation}
\begin{equation}\label{mKdV-**}3a\varphi'(x)+2b\varphi(x)\leq 0\,\,\, \textrm{for} \,\,\,x\in [0,L],\end{equation} and
\begin{equation}\label{mKdV-***}\left\{\begin{array}{l} \theta_1:=2\int\limits_0^L \frac{\left(a\varphi'''(x)+b\varphi''(x)\right)^2}{\varphi'(x)} dx<\infty; \\ \theta_2:=\frac{1}{2}\int\limits_0^L  \frac{{\varphi^2(x)}}{\varphi'(x)}dx<\infty.\end{array}\right.\end{equation}

Multiplying equation \eqref{mKdV-1} by $u(x, t)\varphi(x)$ and integrating by parts, we obtain
\begin{equation}\label{mKdV-3}
\begin{split}
  &\int\limits_0^L\partial^\alpha_{+0,t}u(x,t)u(x,t)\varphi(x)dx \\&=\frac{1}{4}\int\limits_0^Lu^4(x,t)\varphi'(x)dx +\frac{1}{2}\int\limits_0^Lu^2(x,t)(a\varphi'''(x)+b\varphi''(x))dx\\
    & -\int\limits_0^Lu_x^2(x,t)\left(\frac{3a}{2}\varphi'(x)+b\varphi(x)\right)dx+\mathcal{B}(u(x,t),\varphi(x))\Big{|}_0^L,
\end{split}
\end{equation} where \begin{align*}\mathcal{B}(u(x,t),\varphi(x))=&-\frac{1}{4}u^4(x,t)\varphi(x)+u(x,t)u_x(x,t)(a\varphi'(x)+b\varphi(x)) \\&-\frac{u^2(x,t)}{2}(a\varphi''(x)+b\varphi'(x))+\frac{a}{2}u^2_x(x,t)\varphi(x)\\&-au(x,t)u_{xx}(x,t)\varphi(x).\end{align*}
Then we have
\begin{align*}
  \frac{1}{2}&\int\limits_0^Lu^2(x,t)(a\varphi'''(x)+b\varphi''(x))dx +\frac{1}{4}\int\limits_0^Lu^4(x,t)\varphi'(x)dx\\
  &=\frac{1}{4}\int\limits_0^Lv^4(x,t)\varphi'(x)dx-\frac{1}{4}\int\limits_0^L\frac{(a\varphi'''(x)+b\varphi''(x))^2}{\varphi'(x)}dx,
\end{align*} where $$v^2(x,t)=u^2(x,t)+\frac{a\varphi'''(x)+b\varphi''(x)}{\varphi'(x)}.$$

Using the H\"{o}lder inequality and inequality \eqref{Ineq}, we obtain
\begin{align*}&\left(\int\limits_0^Lv^2(x,t)\varphi(x)dx\right)^2 \leq\int\limits_0^Lv^4(x,t)\varphi'(x)dx\int\limits_0^L\frac{\varphi^2(x)}{\varphi'(x)}dx,\end{align*}
\begin{equation*}\partial^\alpha_{+0,t}u(x,t)u(x,t)\varphi(x)\geq \frac{1}{2}\partial_{+0,t}^\alpha \left(u^2(x,t)\varphi(x)\right).\end{equation*}
From \eqref{mKdV-**}, we also get
\begin{equation*}-\int\limits_0^Lu_x^2(x,t)\left(\frac{3a}{2}\varphi'(x)+b\varphi(x)\right)dx\geq 0.\end{equation*}
Then, expression \eqref{mKdV-3} takes the form
\begin{equation}\label{mKdV-4}
\partial_{+0,t}^\alpha F(t)\geq \theta_2^{-1}F^2(t)+\Phi(t)-\theta_1,
\end{equation} where $$F(t)=\int\limits_0^Lv^2(x,t)\varphi(x)dx$$ and $$\Phi(t)=2\mathcal{B}(u(L,t),\varphi(L))-2\mathcal{B}(u(0,t),\varphi(0)).$$
Then the following theorem holds.
\begin{theorem}\label{thmKdV} Let $u_0(x)\in L^1([0,L])$ and the solution $u$ of the equation \eqref{mKdV-1} is such that  $u\in C^{1,3}_{t,x}\left((0,L)\times(0,T)\right)$ and let the function $\varphi$ satisfy conditions \eqref{mKdV-*}, \eqref{mKdV-**} and \eqref{mKdV-***}. If $\Phi(t)-\theta_1\geq 0,\,\,\textrm{for all}\,\, t>0,$ and $F(0)>0,$ then $$F(t)\rightarrow +\infty\,\,\, \textrm{for} \,\,\,t\rightarrow T^*,$$ where $T^*$ satisfies estimate \eqref{T*}.
\end{theorem}
\begin{proof} Since $\Phi(t)-\theta_1\geq 0,\,\,\textrm{for all}\,\, t>0,$ it follows from \eqref{mKdV-4} that
\begin{equation*} \partial_{+0,t}^\alpha \tilde{F}(t)\geq \tilde{F}^2(t),\end{equation*} where $\tilde{F}(t)=\theta_2 F(t).$

As the function $\tilde{F}(t)$ is an upper solution of equation \eqref{01}, therefore $\tilde{F}(t)\rightarrow +\infty$ for $t\rightarrow T^*,$ where $T^*$ satisfies estimate \eqref{T*}. Whereupon $F(t)\rightarrow +\infty$ for $t\rightarrow T^*.$
\end{proof}
\begin{example} Consider problem \eqref{mKdV-1} with $a=2$ and $b=3$ on the interval $[0, 1]$, supplemented with Dirichlet type boundary conditions
\begin{align*} &u(0,t)=0,\,\,t\geq 0,\\& u(1,t)=0,\,\,t\geq 0,\\& u_{x}(1,t)=\sqrt{e}u_x(0,t),\,\,t\geq 0,\end{align*} where $e=\exp(1)$ is Euler's number. Then, if $$\varphi(x)=-\exp(-x),$$ we obtain
$$\theta_1:=0,\,\,\theta_2:=\frac{1-e^{-1}}{2}$$ and \begin{align*} \Phi(t)=\theta_1=0,\,\,\textrm{for all}\,\,\,\, t>0;\end{align*} hence it follows from Theorem \ref{thmKdV} that the solution of problem \eqref{mKdV-1} blows up in finite time under the condition
$$\int\limits_0^1u^2_0(x)\exp(-x)dx<1-e^{-1}.$$
\end{example}
\begin{example}Let $a=0$ and $b>0.$ Let in problem \eqref{mKdV-1} on the interval $[0, 1]$ be given Dirichlet boundary conditions
\begin{align*} &u(0,t)=0,\,\,t\geq 0,\\& u(1,t)=0,\,\,t\geq 0.\end{align*} Then, if $$\varphi(x)=x-1,$$ we obtain
$$\theta_1:=0,\,\,\theta_2:=\frac{1}{6}$$ and \begin{align*} \Phi(t)=\theta_1=0,\,\,\textrm{for all}\,\,\,\, t>0;\end{align*} hence it follows from Theorem \ref{thmKdV} that the solution of problem \eqref{mKdV-1} blows up in finite time under the condition
$$\int\limits_0^1u^2_0(x)(x-1)dx>0,$$ .
\end{example}

\section{Maximum principle and gradient blow-up in time-fractional Burgers equation}
The purpose of this section is to study time-fractional Burgers equation
\begin{equation}\label{Bur-1}
\partial^\alpha_{+0,t}u+uu_x=\nu u_{xx},\,\,\,x\in (0,L),\,t>0,
\end{equation} with the initial condition
\begin{equation}\label{Bur-2}
u(x,0)=u_0(x),\,\,\,x\in[0,L],
\end{equation}
where $\nu>0$ and $u_0$ is a sufficiently smooth function.
\subsection{Maximum principle} In this subsection, we present a maximum principle for the time-fractional Burgers equation \eqref{Bur-1}.
\begin{theorem}\label{Bur-1} Let $u\left( x,t \right)$  satisfy the time-fractional Burgers equation \eqref{Bur-1} with Cauchy data \eqref{Bur-2}. Then $$u\left( x,t \right)\ge \underset{\left( x,t \right)}{\mathop{\min }}\,\{u\left(0, t \right), u\left(L, t \right),u_0\left( x \right)\}\text{ for }\left( x,t \right)\in [0,L]\times[0,T).$$
\end{theorem}
\begin{proof}Let $$m=\underset{\left( x,t \right)}{\mathop{\min }}\,\{u\left(a, t \right), u\left(b, t \right), u_0\left( x \right)\}\text{ }$$ and $$\tilde{u}\left( x,t \right)=u\left( x,t \right)-m.$$ Then, we have \begin{align*}&\tilde{u}\left(0, t \right)=u \left(0, t \right)-m\ge 0,\,\,t\in \left[0, T\right),\\& \tilde{u}\left(L, t \right)=u \left(L, t \right)-m\ge 0,\,\,t\in \left[0, T\right),\end{align*} and $$\tilde{u}\left(x,0 \right)=u_0 \left( x \right)-m\ge 0,\,\,x\in [0,L].$$

Since $$\partial_{+0,t}^{\alpha}\tilde{u}(x,t)=\partial_{+0,t}^{\alpha}u(x,t)$$ and $$\tilde{u}_{xx}\left( x,t \right)=u_{xx}\left( x,t \right),$$
it follows that $\tilde{u}\left( x,t \right)$ satisfies:
$$\partial_{+0,t}^{\alpha}\tilde{u}(x,t)+\tilde{u}(x,t)\tilde{u}_x(x,t)+m\tilde{u}_x(x,t)=\nu\tilde{u}_{xx}\left( x,t \right),$$
and the initial condition
$$\tilde{u}\left(x, 0\right)=u_0 \left(x\right)-m\geq 0,\,x\in [a,b].$$
Suppose that there exits some $\left( x,t \right) \in [0,L]\times[0,T)$  such that $\tilde{u}\left( x,t \right)$ is negative.
Since $$\tilde{u}\left( x,t \right)\ge 0,\,\,\left( x,t \right)\in \{0\}\times \left[0, T \right]\cup\{L\}\times \left[0, T \right]\cup [0,L]\times \{0\},$$
there is $\left( {{x}_{0}},{{t}_{0}} \right) \in (0,L)\times(0,T]$  such that $\tilde{u}\left( {{x}_{0}},{{t}_{0}} \right)$ is the negative minimum of
$\tilde{u}$ over $(0,L)\times(0,T].$ It follows from Property \ref{Pr1} that $\partial^\alpha_{+0,t}u(x_0,t_0)<0.$

Therefore at $\left( {{x}_{0}},{{t}_{0}} \right)$, we get $$\partial^\alpha_{+0,t} \tilde{u}(x_0, t_0)< 0,\,\,\tilde{u}_x(x_0, t_0)=0\,\,\, \textrm{and} \,\,\, \nu \tilde{u}_{xx}\left(x_0,t_0\right)\ge 0.$$ This contradiction shows that $\tilde{u}\left( x,t \right)\ge 0,$ whereupon  $u\left( x,t \right)\ge m$ on $[0,L]\times[0,T]$.
\end{proof}
A similar result can be obtained for a nonpositive solution $u(x, t)$ by considering $-u(x, t).$
\begin{theorem} Suppose that $u\left( x,t \right)$  satisfies \eqref{Bur-1}, \eqref{Bur-2}. Then $$u\left( x,t \right)\leq \max\limits_{(x,t)}\left\{u(L,t), u(0,t), u(x,0)\right\},\,\,\left( x,t \right)\in [0,L]\times[0,T).$$
\end{theorem}

\subsection{Gradient blow-up}
Now suppose that the boundary conditions are set in such a way that the global in time solution of equation \eqref{Bur-1} is bounded. Let there exist a smooth bounded solution $u$ such that $|u(x,t)| \leq M.$ Differentiating equation \eqref{Bur-1} with respect to $x,$ we obtain
\begin{equation}\label{Bur-3}
\partial^\alpha_{+0,t}u_x+uu_{xx}+u^2_x-\nu u_{xxx}=0,\,\,\,x\in (0,L),\,t>0.
\end{equation}
Substituting the expression for $u_{xx}$ from \eqref{Bur-1} into \eqref{Bur-3}, we obtain
\begin{equation}\label{Bur-4}
\partial^\alpha_{+0,t}u_x+u^2_x+\frac{1}{\nu}u\partial^\alpha_{+0,t}u+\frac{1}{\nu}u^2u_x=\nu u_{xxx},\,\,\,x\in (0,L),\,t>0.
\end{equation}
Multiply the equation \eqref{Bur-4} by the function $0\leq\varphi(x)\in C^3([0,L])$ and integrate by parts over the domain $[0,L]$ to get
\begin{equation}\label{Bur-5}\begin{split}&\int\limits_0^L\left(\partial^\alpha_{+0,t}u_x(x,t)+\frac{1}{\nu}u(x,t)\partial^\alpha_{+0,t}u(x,t)\right)\varphi(x)dx\\& =-\int\limits_0^L\left(u^2_x(x,t)+\frac{1}{\nu}u^2(x,t)u_x(x,t)\right)\varphi(x)dx-\nu\int\limits_0^Lu(x,t)\varphi'''(x)dx \\&+\nu\left(u_{xx}(x,t)\varphi(x)-u_x(x,t)\varphi'(x)+u(x,t)\varphi''(x)\right).\end{split}\end{equation}
Denote $v=-u_x-\frac{1}{2\nu}u^2.$ Then, using $|u(x,t)|\leq M$ and inequality \eqref{Ineq}, we obtain
\begin{equation}\label{Bur-6}
\begin{split}\partial^\alpha_{+0,t}\int\limits_0^Lv(x,t)\varphi(x)dx& \geq\int\limits_0^Lv^2(x,t)\varphi(x)dx\\&-\frac{M^4}{4\nu^2}\int\limits_0^L\varphi(x)dx-M\nu\int\limits_0^L|\varphi'''(x)|dx \\&-\nu\left(u_{xx}(x,t)\varphi(x)-u_x(x,t)\varphi'(x)+u(x,t)\varphi''(x)\right)\Big{|}_0^L.\end{split}
\end{equation}
Let
\begin{equation}\label{Bur-7}
\theta_1:=\frac{M^4}{4\nu^2}\int\limits_0^L\varphi(x)dx+M\nu\int\limits_0^L|\varphi'''(x)|dx<\infty,
\end{equation}
\begin{equation}\label{Bur-8}
\theta_2:=\int\limits_0^L\varphi(x)dx<\infty.
\end{equation}
Using H\"{o}lder inequality, we can rewrite the expression \eqref{Bur-6} in the form
\begin{equation}\label{Bur-9}
\partial_{+0,t}^\alpha F(t)\geq \theta_2^{-1}F^2(t)+\Phi(t)-\theta_1,
\end{equation} where $$F(t)=\int\limits_0^Lv(x,t)\varphi(x)dx$$ and $$\Phi(t)=-\nu\left(u_{xx}(x,t)\varphi(x)-u_x(x,t)\varphi'(x)+u(x,t)\varphi''(x)\right)\Big{|}_0^L.$$
Then the following theorem holds.
\begin{theorem}\label{thmBur} Let $u_0(x)\in C^1([0,L])$ and the solution $u$ of the equation \eqref{Bur-1} be such that $u\in C^{1,3}_{t,x}\left((0,L)\times(0,T)\right)$ and let the function $\varphi$ satisfy conditions \eqref{Bur-7} and \eqref{Bur-8}. If $\Phi(t)-\theta_1\geq 0,\,\,\textrm{for all}\,\, t>0,$ and $$F(0)=-\int\limits_0^L\left(u'_0(x)+\frac{1}{2\nu}u^2_0(x)\right)\varphi(x)dx>0,$$ then $$F(t)\rightarrow +\infty\,\,\, \textrm{for} \,\,\,t\rightarrow T^*,$$ where $T^*$ satisfies estimate \eqref{T*}.
\end{theorem}
We do not provide the proof this theorem  as it runs parallel to that of Theorem \ref{thmKdV}.

\section*{Conclusion} In this article, we have studied blowing-up solutions to some time-fractional nonlinear partial differential equations. In precise terms, we have obtained the following results:
\begin{itemize}
  \item Blowing-up solutions of the time-fractional Rosenau-KdV-BBM-Burgers equation with initial conditions;
  \item Blowing-up solutions of the time-fractional Camassa-Holm--Degasperis-Procesi equation with initial conditions;
  \item Blowing-up solutions of the time-fractional Ostrovsky equation with Cauchy data;
  \item Blowing-up solutions of the initial problem for the time-fractional analogue of the modified Korteweg-de Vries-Burgers equation with dissipation;
  \item Maximum principle and gradient blow-up in time-fractional Burgers equation.
\end{itemize}

\section*{Acknowledgements} The research of Ahmad, Alsaedi and Kirane is supported by NAAM research group,  King Abdulaziz University, Jeddah. The research of Torebek is financially supported by the "5--100" program of RUDN University and by a grant No.AP05131756 from the Ministry of Science and Education of the Republic of Kazakhstan. No new data was collected or generated during the course of research

\end{document}